\documentclass[english]{ourlematema}
\usepackage{amsmath, amsthm, amssymb, hyperref, color}
\usepackage{MnSymbol} % for \dashedrightarrow command
\usepackage{graphicx}
\usepackage{caption}
\usepackage{verbatim}
\usepackage{chemfig, chemnum}
\usepackage{tikz}
\usepackage[linesnumbered,lined,commentsnumbered]{algorithm2e}
\usepackage{caption}
\usepackage{subcaption}
\usepackage{makecell}
\usepackage{array}
\usepackage{booktabs}

\newcommand{\RR}{\mathbb{R}}

\renewcommand{\S}{\mathbb{S}}

\newcommand {\cS} {\mathcal {S}}
\newcommand {\cL} {\mathcal {L}}

\newcommand{\<}{\langle}
\renewcommand{\>}{\rangle}

 \title{Likelihood geometry of correlation models}
 
  \author{Carlos Am\'endola}
  \address{%
  Technische Universit\"at M\"unchen, Universit\"at Ulm \\
\email{carlos.amendola@tum.de}
}

\author{Piotr Zwiernik}
\address{%
Universitat Pompeu Fabra, Barcelona \\
\email{piotr.zwiernik@upf.edu}
}

\date{2020/11/11}

\begin{document}
\maketitle
\begin{abstract}
\noindent
 Correlation matrices are standardized covariance matrices. They form an affine space of symmetric matrices defined by setting the diagonal entries to one. We study the geometry of maximum likelihood estimation for this model and linear submodels that encode additional symmetries. We also consider the problem of minimizing two closely related functions of the covariance matrix: the Stein's loss and the symmetrized Stein's loss. Unlike the Gaussian log-likelihood these two functions are convex and hence admit a unique positive definite optimum. Some of our results hold for general affine covariance models.
\end{abstract}

\section{Introduction}

Learning a covariance matrix is arguably one of the most basic problems in statistics. Learning structured covariance matrices has many applications ranging from signal processing to high-dimensional statistics. A type of structure that is most relevant in this article is given by affine constraints. More specifically, we focus on the problem of estimating a correlation, that is, a positive definite matrix $\Sigma$ such that the diagonal entries are $1$. This problem has appeared in statistics in many different contexts  \cite{brownlees2020projected,higham2002computing,liang1986longitudinal,rousseeuw1994shape,small2000eliminating,kendall1961advanced}.  

The literature on estimating the correlation matrix focuses on two canonical problems. In the first, the data are represented by a symmetric positive semidefinite (PSD)  matrix $S$ and the goal is to estimate the correlation matrix $\Sigma$ by finding a suitable projection of $S$ on the space of correlation matrices; see e.g.  \cite{brownlees2020projected,ZUR}. In the second problem, each entry of $\Sigma$ can be estimated from a different sample and, in consequence, the data is represented by a symmetric but not necessarily PSD matrix \cite{higham2002computing}. In this paper we focus on the first scenario assuming  that $S$ is positive definite. 

The maximum likelihood estimation problem in this context has been studied recently in \cite[Section~4]{ZUR} from an optimization and probabilistic point of view. In this paper we analyze it through the lenses of algebra and geometry. 

The likelihood geometry of the bivariate correlation model has been analyzed in some detail; see \cite[Example 18.3]{kendall1961advanced}, \cite[Example 2.38]{barndorff1994inference}, and \cite[Section~2.1]{small2000eliminating}. In this bidimensional case, estimating the correlation parameter reduces to solving a cubic equation, so that the likelihood can have either one or three (real) critical points. To distinguish between these two situations it is enough to study the discriminant of the cubic polynomial. In this paper we perform a careful analysis of the discriminant locus to identify not only when there are three critical points, but also when all three points correspond to positive definite solutions (Theorem \ref{th:bivariatecritical}). We provide a similar study of the second derivative of the likelihood function to discover a rather surprising fact: for most statistically relevant data points, the likelihood function is concave (Theorem \ref{th:bivariateconvexity}). Moreover, restricting to the equicorrelation model we are able to generalize the results on the number and nature of the critical points of the likelihood to any dimension. In particular, estimating the equicorrelation parameter reduces to solving a generalized cubic equation (Theorem \ref{th:MLESigequicorr}).

A similarly detailed analysis is typically not possible for general linear correlation models. For the unrestricted correlation model we show that even in low-dimensional cases, the likelihood function can have many positive definite critical points. There are, however, some unexpected and positive features that are common to all linear correlation models. Although the likelihood function is typically not a concave function, we show that for linear correlation models there are always some data points for which this function is globally concave (Proposition~\ref{prop:convnonempty}). This is in sharp contrast to the general covariance model. 
As an alternative to the non-convex problem of optimizing the likelihood function, we propose minimizing two other functions that are also  popular in covariance matrix estimation: the Stein's loss and the symmetrized Stein's loss. We show that, like the former, the latter is strictly convex (Proposition~\ref{prop:convexSS}). All three functions appear in statistics in the context of likelihood-based inference for Gaussian data and they also lead to pseudolikelihood estimators for a wide range of non-Gaussian scenarios. Minimizing the Stein's loss in the correlation model was recently considered in \cite{brownlees2020projected}, or more generally, as a Bregman projection on affine models in \cite{dhillon2008matrix,kulis2009low}. 

The paper is organized as follows. In Section~\ref{sec:3losses} we introduce the three loss functions we consider over the space of correlation matrices. We analyze in detail the  bivariate correlation model in Section~\ref{sec:bivariate}, which serves as the basis for the study that follows. In Section~\ref{sec:unrestricted} we look at the likelihood geometry of the most general unrestricted correlation model, and provide a simple efficient algorithm to optimize Stein's loss. On the other extreme, we consider the simplest one dimensional model, the equicorrelation model, in Section~\ref{sec:equicorr}. We show that it exhibits many similarities with the bivariate case. Finally, we conclude in Section~\ref{sec:further} with natural follow-up questions and possible future directions.

\section{Three loss functions}\label{sec:3losses}

Denote by $\S^n$ the set of $n\times n$ symmetric matrices and by $\S^n_+$ its subset given by all positive definite matrices. Define the divergence function $I:\S^n_+\times \S^n_+\to \RR$
\begin{equation}\label{eq:Idiv}
I(\Sigma_1\| \Sigma_2)\;=\;\<\Sigma_1,\Sigma_2^{-1}\>-\log\det(\Sigma_1 \Sigma_2^{-1})-n,
\end{equation}
where $\<\Sigma_1,\Sigma_2^{-1}\>={\rm trace}(\Sigma_1 \Sigma_2^{-1})$ is the standard inner product on $\S^n$. Note that $I(\Sigma_1\| \Sigma_2)$ is convex both in $\Sigma_1$ and $K_2=\Sigma_2^{-1}$ (but not in $\Sigma_2$). Fenchel duality also implies that $I(\Sigma_1\| \Sigma_2)\geq 0$ with equality if and only $\Sigma_1=\Sigma_2$.

Based on this divergence function we define two functions. For a fixed $S\in \S^n_+$, the \emph{entropy loss} is $I(S||\Sigma)$. Similarly,  the \emph{Stein's loss} is defined as $I(\Sigma\|S)$. The Stein's loss is a convex function of $\Sigma$ but the entropy loss is not. Up to constant factors, the entropy loss function is the negative of the Gaussian log-likelihood and the Stein's loss is the negative dual likelihood; see \cite{sturmfels2020estimating}. For this reason the optimum of the entropy loss is called here the maximum likelihood estimator (MLE) and the optimum of the Stein's loss is called the dual MLE. 

Since the divergence is not symmetric, it is useful to define the \emph{symmetrized Stein's loss}: 
$$
L(\Sigma,S)\;=\;\frac{1}{2}\Big(I(S||\Sigma)+I(\Sigma||S)\Big)\;=\;\frac{1}{2}\Big(\<S,\Sigma^{-1}\>+\<\Sigma,S^{-1}\>\Big)-n.
$$
The symmetrized Stein's loss has been recently used in the context of high-dimensional matrix estimation, e.g. \cite{ledoit2018optimal,soloff2020covariance}. Following \cite{moakher2006symmetric} we list some basic properties of the symmetrized Stein's loss:
\begin{enumerate}
    \item [(i)] $L(\Sigma,S)$ is nonnegative, and zero if and only if $\Sigma=S$.
    \item [(ii)] $L(\Sigma,S)=L(S,\Sigma)$,
    \item [(iii)] $L(\Sigma,S)=L(\Sigma^{-1},S^{-1})$,
    \item [(iv)] $L(\Sigma,S)=L(P^T\Sigma  P,P^T S P)$ for all invertible $P\in \RR^{n\times n}$.
\end{enumerate}

These properties made the symmetrized Stein's loss a convenient tool for theoretical analysis in statistics. As we argue in this paper, it can also offer a valuable and numerically tractable alternative to both the maximum likelihood estimator and the dual likelihood estimator. This observation is based on the following, rather surprising, result. 

\begin{prop}\label{prop:convexSS}
The symmetrized Stein's loss $L(\Sigma,S)$ is a strictly convex function both in $\Sigma\in \S^n_+$ and in $K=\Sigma^{-1}$.
\end{prop}
\begin{rem}
In the following proof and throughout we use the standard expressions for directional derivatives of functions of $\Sigma$. Denoting by $\nabla_U$ the directional derivative in the direction $U\in \S^n$ we have $\nabla_U \<\Sigma,S^{-1}\>=\<U,S^{-1}\>$ and  
$$
\nabla_U \<S,\Sigma^{-1}\>\;=\;\<S,\nabla_U \Sigma^{-1}\>\;=\;-\<S,\Sigma^{-1}U\Sigma^{-1}\>.
$$
Moreover,
$$
\nabla_U \log\det\Sigma\;=\;\<\Sigma^{-1},U\>.
$$
From this it immediately follows that the gradient of $\<S,\Sigma^{-1}\>$ is $-\Sigma^{-1}S\Sigma^{-1}$ and the gradient of $\log\det\Sigma$ is $\Sigma^{-1}$.
\end{rem}
\begin{proof}[Proof of Proposition~\ref{prop:convexSS}]
We show that $L(\Sigma,S)$ is a strictly convex function in $\Sigma$. The proof that it is also strictly convex in $K=\Sigma^{-1}$ is analogous. Since $\<\Sigma,S^{-1}\>$ is linear in $\Sigma$, it is enough to show that $\<S,\Sigma^{-1}\>$ is strictly convex. We check the second order conditions for convexity. The directional derivative in the direction $U\in \S^n$ satisfies
$$
\nabla_U \<S,\Sigma^{-1}\>\;=\;-\<S,\Sigma^{-1}U\Sigma^{-1}\>.
$$
Computing the directional derivative of $\nabla_U L(\Sigma,S)$ in the direction $V\in \S^n$ gives
$$
\nabla_V\nabla_U \<S,\Sigma^{-1}\>\;=\;\<S,\Sigma^{-1}V\Sigma^{-1}U\Sigma^{-1}\>+\<S,\Sigma^{-1}U\Sigma^{-1}V\Sigma^{-1}\>.
$$
The function $\<S,\Sigma^{-1}\>$ is strictly convex if and only if $\nabla_U\nabla_U \<S,\Sigma^{-1}\>> 0$ for all $U\in \S^n$, $U\neq 0$. We have
$$
\nabla_U\nabla_U \<S,\Sigma^{-1}\>\;=\;2\<S,KUKUK\> 
$$
where $S$ is positive definite and $KUKUK$ is positive semidefinite. It follows that  $\nabla_U\nabla_U \<S,\Sigma^{-1}\>\geq 0$ and $\nabla_U\nabla_U \<S,\Sigma^{-1}\>=0$ if and only if $KUKUK=0$ (equiv. $UKU=0$). It remains to show that, for any positive definite $K$ and symmetric $U$, $UKU=0$ if and only if $U=0$. But this is clear. If $U\neq 0$ then there exists $x$ such that $y=Ux\neq 0$. But then $x^T UKUx=y^T K y>0$ and so $UKU\neq 0$.  
\end{proof}

All three loss functions are continuously differentiable with respect to $\Sigma$, making them suitable to optimization by standard techniques. Their gradients take a simple form. We consistently use the notation $K= \Sigma^{-1}$ and $W = S^{-1}$. 

\begin{prop}\label{prop:gradients}
Let $S \in \S^+_n$ be fixed. Then the gradient of the entropy loss is 
$$
\nabla I(S\| \Sigma)\;=\; K - KSK.
$$
The gradient of the Stein's loss is 
$$
\nabla I(\Sigma\| S)\;=\; W - K.
$$
The gradient of the symmetrized Stein's loss is 
$$
\nabla L(\Sigma,S)\;=\; \frac{1}{2}(W - KSK).
$$
\end{prop}
\begin{proof}
The expression for the gradient of $L(\Sigma,S)$ follows from the proof of Proposition~\ref{prop:convexSS}. The other two cases can be verified in a similar way.
\end{proof}

Denote by $\S^n_0$ the set of all $n\times n$ symmetric matrices with zeros on the diagonal. The set of correlation matrices is precisely $(I_n+\S^n_0)\cap \S^n_+$, where $I_n$ is the identity matrix. Given a vector subspace $\cL\subseteq \S^n_0$ we consider the linear correlation model
\begin{equation}\label{eq:corrmodel}
(I_n+\cL)\cap \S^n_+.
    \end{equation}
Given a fixed matrix $S\in \S^n_+$, the goal is to optimize the entropy loss, the Stein's loss, and the symmetrized Stein's loss over this model.

The following result gives equations describing the critical points in the problem of minimizing the three loss functions over any spectrahedron $(A+\cL)\cap \S^n_+$. The case of correlation models is recovered by taking $A=I_n$ and $\cL\subseteq \S^n_0$.
\begin{thm}\label{th:critical}
Consider the model obtained by restricting $\S^n_+$ to an affine subspace $A+\cL$ for $A\in \S^n_+$ and a linear subspace $\cL\subseteq \S^n$. A necessary condition for $\Sigma\in \S^n$ to be a local minimum of the entropy loss is that 
\begin{equation}\label{eq:criticalML}
\Sigma-A\in \cL\quad \mbox{ and }\quad K-KSK\in\cL^\perp.
\end{equation}
A necessary condition for $\Sigma$ to be a local minimum of the Stein's loss
\begin{equation}\label{eq:criticalS}
\Sigma-A\in \cL\quad \mbox{ and }\quad K-W\in\cL^\perp.
\end{equation}
 A necessary condition for $\Sigma$ to be a local minimum of the symmetrized Stein's loss
\begin{equation}\label{eq:criticalSS}
\Sigma-A\in \cL\quad \mbox{ and }\quad W-KSK\in\cL^\perp.
\end{equation}
If $\Sigma\in \S^n_+$, the last two conditions are also sufficient.
\end{thm}
\begin{proof}
All three loss functions are continuously differentiable and the model $(A+\cL)\cap \S^n_+$ is relatively open in $A+\cL$. Therefore, a necessary condition for $\Sigma\in A+\cL$ to be a local minimum is that all directional derivatives in the directions $U\in \cL$ are zero. Conditions (\ref{eq:criticalML}), (\ref{eq:criticalS}), and (\ref{eq:criticalSS}) then follow from the gradient formula in Proposition~\ref{prop:gradients}. The last two conditions are also sufficient if $\Sigma\in \S^n_+$ because  these optimization problems are convex, see Proposition~\ref{prop:convexSS}.  
\end{proof}

Since the Stein's loss and the symmetrized Stein's loss are convex functions of $\Sigma$, these two functions admit a unique optimum $(I_n+\cL)\cap \S^n_+$. However, the entropy loss may admit several local optima. In all cases there will typically be many complex critical points. Given the algebraic nature of the critical equations, this number famously does not depend on $S$ if it is generic.
\begin{dfn}
Let $S$ be a generic covariance matrix. The \emph{ML degree} is the number of complex solutions to \eqref{eq:criticalML}, the \emph{dual ML degree} the number of complex solutions to \eqref{eq:criticalS}, and the \emph{SSL degree} the number of complex solutions to \eqref{eq:criticalSS}.
\end{dfn}

While the ML degree and dual ML degree\footnote{The term \textit{reciprocal ML degree} has been used recently to refer to the dual ML degree \cite{pencils, boege2020reciprocal}. We prefer \textit{dual} from an optimization standpoint, following the original definition from Efron \cite{efron1978geometry}.} for linear covariance models have been studied before in algebraic statistics (see \cite{sturmfels2020estimating, coons2020maximum}), the SSL degree we propose is novel and it should be interesting to study it further. Note that, unlike the function $I(\Sigma_1\| \Sigma_2)$ whose critical points give rise to the first two degrees, the symmetrized Stein's loss $L(\Sigma,S)$ is itself an algebraic function.

\section{The bivariate correlation model}\label{sec:bivariate}

In this section we study in detail the bivariate correlation model. Thus, $n=2$ and the model consists of all matrices of the form
$$
\Sigma=\begin{bmatrix}
1 & \rho\\
\rho & 1
\end{bmatrix}\qquad\mbox{for }\rho\in (-1,1).
$$
The problem of optimizing the entropy loss over this model is relatively well studied; see \cite[Example 18.3]{kendall1961advanced}, \cite[Example 2.38]{barndorff1994inference}, and \cite[Section~2.1]{small2000eliminating}. Here we largely extend the existing results. 

\subsection{Critical points of the entropy loss}

The derivative of the entropy loss $I(S\| \Sigma)$ with respect to $\rho$ is $\tfrac{2}{(1-\rho^2)^2} f(\rho)$, where
$$
f(\rho)\;=\;\rho(\rho^2-1)-S_{12}(1+\rho^2)+\rho(S_{11}+S_{22}).
$$
This polynomial has three roots, two of which may be complex. Note that $f(-1)=-2S_{12}-S_{11}-S_{22}<0$ and $f(1)=-2S_{12}+S_{11}+S_{22}>0$ (since $S$ is positive definite) and hence $f$ has at least one root in the open interval $(-1,1)$. As a consequence, if $f(\rho)$ has a unique root in $\RR$, this unique root corresponds to a positive definite correlation matrix.

The polynomial $f(\rho)$ depends on $S$ only through $a=\tfrac{S_{11}+S_{22}}{2}$ and $b=S_{12}$ and with this notation we have
\begin{equation}\label{eq:cubic}
f(\rho)\;=\;\rho^3-b\rho^2-\rho(1-2a)-b.
\end{equation}
The condition that $S$ is positive definite implies that $a>|b|$. Following \cite{kendall1961advanced} we use the discriminant to get the conditions on $S$ that ensure there is a single critical point. The discriminant is
\begin{equation}\label{discriminantfor2}
 \Delta_f(b,a) = -4[ b^4 - (a^2+8a-11)b^2+(2a-1)^3]
\end{equation}
and $f(\rho)$ has a single root in $\RR$ if and only if $\Delta_f(b,a)<0$. When $\Delta_f(b,a)=0$ there is an additional double root and there are three real roots when $\Delta_f(b,a)>0$.

\begin{figure}[ht]
 \centering 
 \includegraphics[width=10cm]{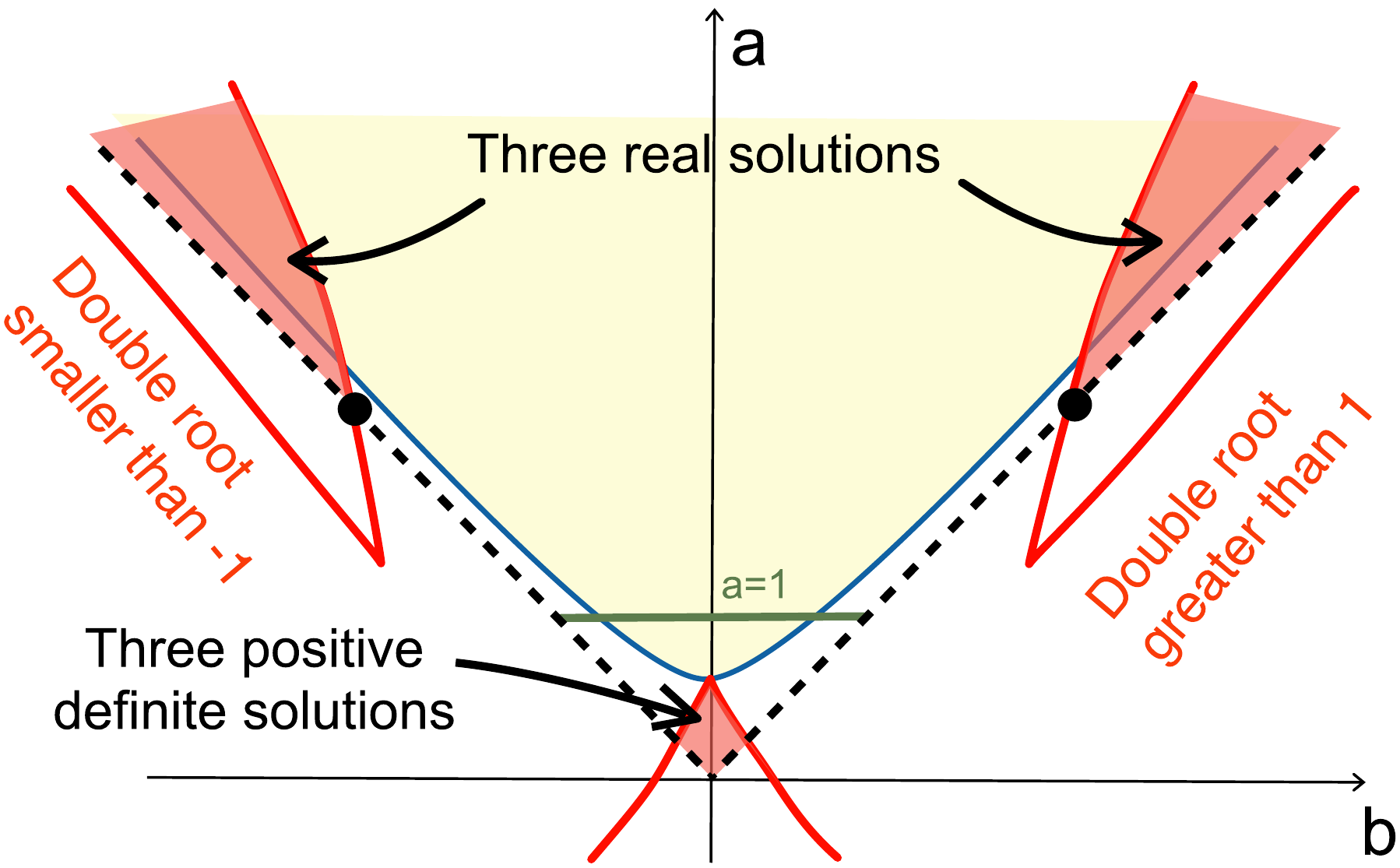} 
\caption{Likelihood geometry for the bivariate correlation model. The boundary of the positive definite cone is given by $a = \vert b \vert$ (dashed black line). The zero set of the discriminant $\Delta_f(b,a)$ is given by a red curve. In the red regions, the cubic in (\ref{eq:cubic}) has three real solutions. The yellow shaded region corresponds to $(b,a)$ for which the entropy loss is convex; see Figure~\ref{fig:convexity}.}\label{fig:realsolutions} 
\end{figure}

To add to the standard discussion of this example we note that even when $f(\rho)$ has three real roots, the other two roots may not lie in $(-1,1)$ in which case the entropy loss still has a unique optimum in the constrained region. To get the full understanding of the likelihood geometry we depict the situation in Figure~\ref{fig:realsolutions}. The red curves represent the zero locus of the discriminant $\Delta_f(b,a)$. This zero locus divides the cone $a>|b|$ into four chambers: the central region where the discriminant is negative and three red regions where it is positive. All $(b,a)$ that lie in the central chamber lead to a unique real critical point: the MLE. The points on the discriminant locus correspond to an additional double root. Since $f(1)=2(a-b)$, $f$ has a root at $\rho=1$ if and only if $a=b$. This shows that the right top part of the discriminant corresponds to a double point in $(1,\infty)$. Indeed, the double point depends continuously on $(b,a)$ and the only point where it could be equal to $1$, namely when $a=b=3+2\sqrt{2}$ (the right black dot in Figure \ref{fig:realsolutions}), it is equal to $\rho = 1 + \sqrt{2} > 1$. Similarly, the fact that $f(-1)=-2(a+b)$ implies that the left component of the discriminant corresponds to double roots in $(-\infty,-1)$. At the point where it could be equal to $-1$, $a=-b=3+2\sqrt{2}$, it is actually equal to $\rho = -1-\sqrt{2} < -1$. Finally, the central component of the discriminant corresponds to double points $\rho \in (-1,1)$. The two points where the discriminant locus intersects the boundary point are $a=b=3-2\sqrt{2}$ and $a=-b=3-2\sqrt{2}$. These correspond to the double roots $\rho=1-\sqrt{2}$ and $\rho=-1+\sqrt{2}$ respectively. This completes the picture of the situation.

\begin{thm}\label{th:bivariatecritical}
For the bivariate correlation model, there is at least one positive definite critical point of the entropy loss function. There are three positive definite critical points if and only if $\Delta_f(b,a)>0$ and $a<\frac12$ (corresponding to the bottom red region in Figure~\ref{fig:realsolutions}).   
\end{thm}

\begin{proof}
Since $S$ is positive definite we have that $a>|b|$. Since $f(-1)<0$ and $f(1)>0$, it is clear that $f$ has always at least one root in $(-1,1)$. As noted above, we get a unique real root (and hence a unique root in (-1,1)) when $(b,a)$ lies in the central chamber of the cone $a>|b|$. Crossing the discriminant into the right red region introduces two extra real roots in $(1,\infty)$. Since the roots of the cubic are continuous in $(b,a)$, we are guaranteed that none of them moves to $(-1,1)$ unless we cross the dashed line. By the same continuity argument, we see that the left red region has exactly one solution in $(-1,1)$ and that the bottom red region corresponds to exactly three solutions in $(-1,1)$.
\end{proof}

\subsection{Convexity of the entropy loss}

Even when the entropy loss has a unique minimum, it does not need to be convex for $\rho\in (-1,1)$. To understand better the complexity of this optimization procedure we now study precisely when the entropy loss is convex.  We do it by studying the second order condition for convexity. The second derivative of the entropy loss with respect to $\rho$ is equal to $\tfrac{2}{(1-\rho^2)^3}g(\rho)$, where
$$
g(\rho)=2a-1   - 6b\rho + 6a\rho^2 - 2b\rho^3 + \rho^4.
$$
The entropy loss is convex if and only if $g(\rho)\geq 0$ for $\rho\in (-1,1)$. We have $g(-1)=8(a+b)$ and $g(1)=8(a-b)$. Both quantities are strictly positive in the cone $a>|b|$ and they can be zero only if $a=-b$ and $a=b$ respectively. 

As for the polynomial $f$, we get a full picture by careful analysis of the discriminant locus. The discriminant $\Delta_g(b,a)$ of $g(\rho)$ is
\begin{eqnarray}\label{eq:discpos}
\Delta_g\;:= \;-256\Big( 27b^6 -27(2a^2+6a-5)b^4+ 9(3a^4+36a^3-32a^2+8a+1)b^2 \nonumber \\[-0.1in] 
-(2a-1)(9a^2-2a+1)^2 \Big)
\end{eqnarray}

\begin{figure}[ht]
 \centering 
 \includegraphics[width=10cm]{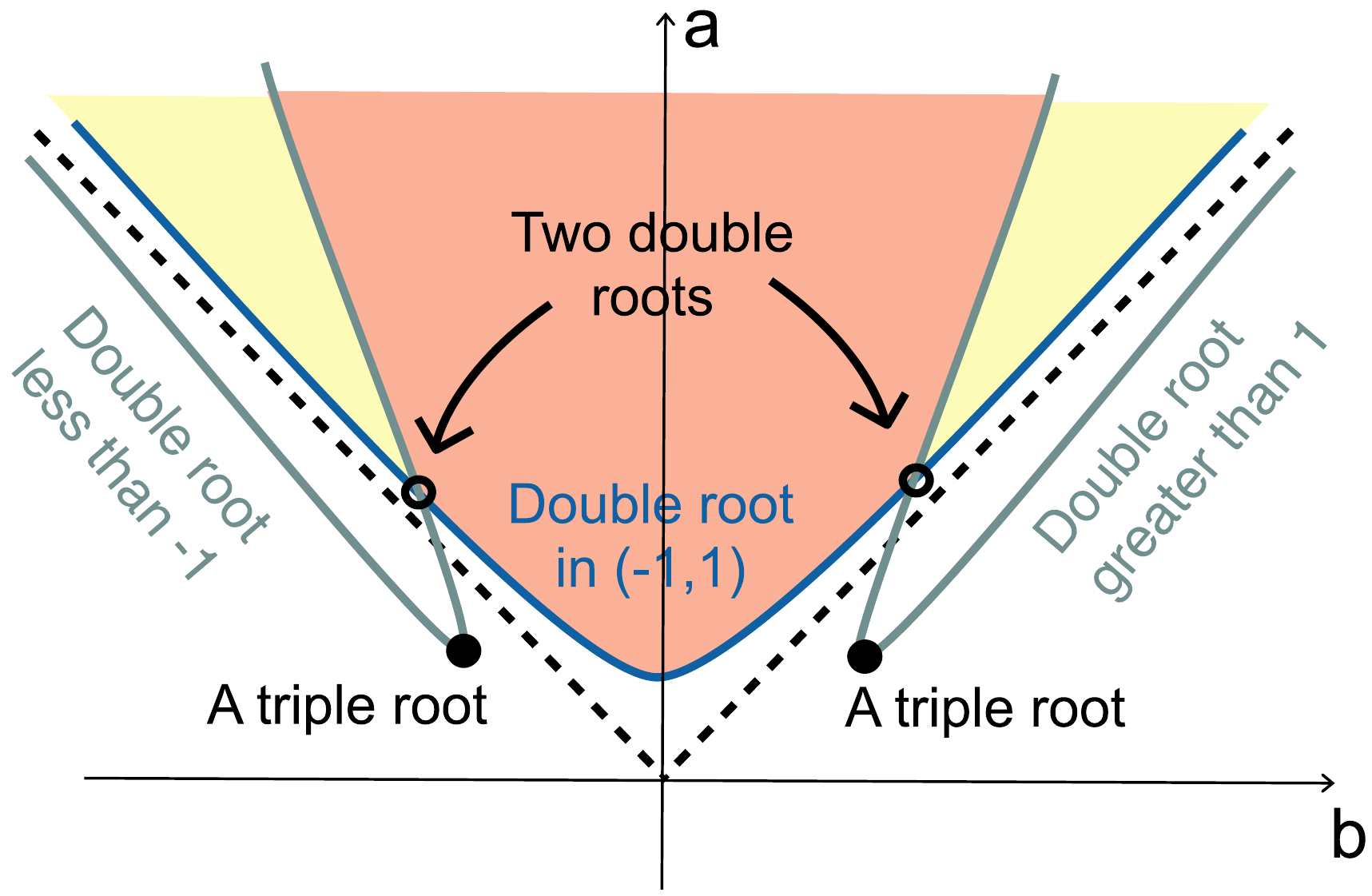} 
\caption{Likelihood geometry for the bivariate correlation model. To the plot of the cone in Figure~\ref{fig:realsolutions} we add in the zero set of the sextic discriminant $\Delta_g(b,a)$. It is given by the blue curve (the central component) and the two green curves.} \label{fig:convexity} 
\end{figure}
The somewhat complex boundary where the sextic $\Delta_g(b,a)=0$ vanishes is sketched in Figure~\ref{fig:convexity}. The \emph{central component of the discriminant} is given by the blue line that crosses the $a$-axis at $(b,a)=(0,\tfrac{1}{2})$. Two extra components that intersect the graph of $a=|b|$ (dashed line) are given in green. The points $(b,a)$ where the discriminant is non-negative, given by the reddish cone in the center, corresponds to the situation when $g(\rho)\geq 0$ for all $\rho\in \RR$. In this case the entropy function is of course convex in $(-1,1)$ but convexity still holds under weaker conditions. The discriminant locus corresponds to double roots of the equation $g(\rho)=0$. Since $g(1)=8(a-b)$, $g$ can have a root at $\rho=1$ if and only if $a=b$. This implies that the upper right branch of the discriminant locus corresponds to double roots in $(1,\infty)$. We check this in the same way as earlier for the polynomial $f$ using the fact that double roots are continuous in $(b,a)$ and the point $(b,a)$ on the intersection of the right branch of the discriminant locus on the dashed line corresponds to a double root greater than 1. Similarly, $g(-1)=8(a+b)$ is zero if and only if $a=-b$, which implies that the left upper branch of the discriminant locus corresponds to a double root in $(-\infty,-1)$. The central component of this locus corresponds to a double root in $(-1,1)$.
 
\begin{thm}\label{th:bivariateconvexity}
The entropy function in the bivariate correlation model is convex over $(-1,1)$ if and only $(b,a)$ lies in the region above the central component of the discriminant $\Delta_g(b,a)$ (corresponding to the union of the reddish and two yellow regions in Figure~\ref{fig:convexity}). 
\end{thm}
\begin{proof}
Since $g(\rho)$ is globally positive for $(b,a)$ inside the central chamber determined by the components (the reddish region in the figure) and because the roots are continuous in $(b,a)$, crossing one of the green lines can introduce roots only in $(-\infty,-1)$ or in $(1,+\infty)$. We conclude that $g(\rho)$ has no root in $(-1,1)$ unless we are below the central component. In particular, it remains always positive in the interval $(-1,1)$ and this is enough to guarantee convexity. Crossing the central component of the discriminant locus introduces two roots $-1<\rho_1<\rho_2<1$. Then $g$ is negative in $(\rho_1,\rho_2)$ and therefore the entropy loss is not convex in this region. 
\end{proof}

We conclude this section by noting that the analysis is much simpler for the Stein's loss, where the unique maximizer is given explicitly. We study this and the symmetrized Stein's loss in Section~\ref{sec:equicorr} as the special case of the equicorrelation model. There we show that the dual ML degree is 2 and give the formula for the optimum of the Stein's loss. We also show that the SSL degree is 4. We find this last number to be surprising. It indicates that, although the symmetrized Stein's loss is convex (and hence easier to handle from an optimization point of view), from an algebraic standpoint it is more complicated.

\section{The unrestricted correlation model}\label{sec:unrestricted}

The unrestricted model corresponds to the case when $\cL=\S^n_0$ in (\ref{eq:corrmodel}). Recall that we denote $K=\Sigma^{-1}$ and $W=S^{-1}$. By Theorem~\ref{th:critical}, any optimal correlation matrix must satisfy:
\begin{enumerate}
    \item [(a)] for the entropy loss: $K_{ij}=(KSK)_{ij}$ for $i\neq j$,
    \item [(b)] for the Stein's loss: $K_{ij}=W_{ij}$ for $i\neq j$,
    \item [(c)] for the symmetrized Stein's loss:  $(KSK)_{ij}=W_{ij}$ for  $i\neq j$.
\end{enumerate}

Solving these equations leads to high-degree polynomial systems. Some results on the ML degree, dual ML degree, and SSL degree for small $n$ are reported in Table~\ref{tab:MLdegcorr}. The case $n=2$ was covered in great detail in Section~\ref{sec:bivariate}. If $n=3$ the dual ML degree is 5, the ML degree is 15, and the SSL degree is 28. Overall, the table indicates that the SSL degree increases faster than the ML degree, which increases faster than the dual ML degree. We currently have no conjectures on a general formula for these sequences. The numbers for $n>4$ were computed using \texttt{LinearCovarianceModels.jl} as introduced in \cite{sturmfels2020estimating}.

\begin{table}[h]
	\begin{tabular}{@{}l| cccccc ccccc @{}}
		  $n$ & 1 & 2 & 3 & 4 & 5 & 6 & 7 & 8 & 9  \\ \midrule
		  	SSL degree   & 1 & 4 & 28 & 292 & ? & ? & ? & ? & ?   \\ \midrule
		ML degree   & 1 & 3 & 15 & 109 & 1077 & 13695 & ? & ? & ?   \\ \midrule
		dual ML degree & 1 & 2 & 5 & 14 & 43 & 144 & 522 & 2028 & 8357   \\
			\bottomrule   
	\end{tabular}  
\caption{The different algebraic degrees for the space of correlation matrices. }\label{tab:MLdegcorr}
	\end{table}

\subsection{Likelihood geometry}

\paragraph{Positive definite critical points} The entropy loss may lead to many local optima, which can substantially complicate the maximum likelihood estimation. We provide basic analysis moving directly to the case $n=3$ with the ML degree equal to 15. To obtain a better understanding of the multi-modality of the entropy loss, we study what happens to the MLE and critical points when we observe the ray $S = t I_n$ (multiple of the identity matrix). Denote by $(s_{12},s_{13},s_{23})$ the coordinates of the linear space $I_3+\S^3_0$. The 15 critical points are exactly:
\begin{enumerate}
    \item [(a)] The origin $(0,0,0)$ (i.e. the identity matrix)
    \item [(b)] The six points $(\pm \sqrt{1-2t}, 0,0)$, $(0,\pm \sqrt{1-2t},0)$ and $(0,0,\pm \sqrt{1-2t})$.
    \item [(c)] The four points $(\alpha,-\alpha,\alpha), (-\alpha,-\alpha,\alpha), (\alpha,\alpha,-\alpha),(-\alpha,-\alpha,-\alpha)$. 
    \item [(d)] The four points $(\beta,-\beta,\beta), (-\beta,-\beta,\beta), (\beta,\beta,-\beta),(-\beta,-\beta,-\beta)$. 
\end{enumerate}
where
    $$\alpha = \frac{t-1 + \sqrt{t^2-18t+9}}{4} \quad \beta = \frac{t-1 - \sqrt{t^2-18t+9}}{4}. $$

First we see the nature of these critical points. All of them are real when $0<t<\frac{1}{2}$. The roots of $t^2-18t+9$ are $t= 3(3\pm 2\sqrt{2)}$ and the points in (c) and (d) are real whenever $0<t<3(3-2\sqrt{2})$ or $t>3(3+2\sqrt{2})$. However, if we want them all to be positive definite, we need to restrict to the first interval. Thus the number of positive definite critical points is $15$ when $0<t< 1/2$. When $t=1/2$ seven of them collapse to the origin which gives $9$ positive definite critical points which is then true for all $1/2\leq t<3(3-2\sqrt{2})$. When $t=3(3-2\sqrt{2})$ the eight points in (c) and (d) collapse to four points giving 5 positive definite critical points. For $t>3(3-2\sqrt{2})$ we have a single positive definite solution. 

By evaluating each of these critical points in the entropy loss, we obtain expressions that we can compare for every $t>0$. All the points in the same group (a), (b), (c), (d) or (e) evaluate to the same likelihood value. Thus, for example, when $t<\frac12$, we see that the points in (d) give the MLE. On the other hand, when $t>3(3-2\sqrt{2)}$, then the origin is the MLE.

\paragraph{Convexity of the entropy function} The entropy function is, in general, not convex. In the bivariate case we have seen however that convexity over the set of correlation matrices $I_n+\S^n_0$ happens for many data points. This motivates the following definition.
\begin{dfn}\label{df:convcone}
For an affine model $M=(A+\cL)\cap \S^+_n$, its \emph{convexity cone} $C(M)$ is the set of all $S\in \S^n_+$ for which $I(S\|\Sigma)$ is a convex function on $M$.
\end{dfn}

In the bivariate correlation model the convexity cone is described in Theorem~\ref{th:bivariateconvexity}. The name of the convexity cone is justified by the following result.
\begin{prop}\label{prop:convcone}
Fix an affine model $M\subseteq \S^n_+$. Then $C(M)$ is  closed under multiplications by a scalar $t\geq 1$.
\end{prop}
\begin{proof}
Suppose $S\in C(M)$. If $t>0$ then
$$
I(tS\|\Sigma)\;=\;tI(S\|\Sigma)+(t-1)(n-\log\det(S^{-1}\Sigma))-n\log t.
$$
On the right hand side the only two terms that depend on $\Sigma$ are $tI(S\|\Sigma)$ and $-(t-1)\log\det(\Sigma)$. The first is convex over $M$ by assumption whenever $t>0$. The second is convex over $\S^n_+$ whenever $t\geq 1$. 
\end{proof}

\begin{rem}\label{rem:convcone}
For the unrestricted covariance model $M=\S^n_+$ the convexity cone $C(M)$ is empty. By \cite[Proposition~3.1]{ZUR}, $I(S\|\Sigma)$ is convex in and only in the region where $2S-\Sigma$ is positive definite and so there is no $S$ for which $I(S\|\Sigma)$ is convex over $\S^n_+$. So a first interesting question is to characterize models $M$ for which $C(M)$ is nonempty.
\end{rem}

Regarding the question in Remark~\ref{rem:convcone}, the following result shows that the convexity cone is always non-empty for any linear correlation model.
\begin{prop}\label{prop:convnonempty}
If $S=t I_n$ for $t\geq n/2$ then $I(S\| \Sigma)$ is convex over the set of all correlation matrices. In particular, for every linear correlation model $M$, its convexity cone is non-empty.
\end{prop}
\begin{proof}
In light of Proposition~\ref{prop:convcone} it is enough to show that the statement holds for $S=\tfrac{n}{2} I_n$. We will verify the second order conditions for convexity. By Proposition~\ref{prop:gradients}, $\nabla_U I(S\|\Sigma)=\<K-KSK,U\>$ for any $U\in \S^n$. Similarly, if $V\in \S^n$ then 
$$
\nabla_V\nabla_U I(S\|\Sigma)\;=\;\<\nabla_V(K-KSK),U\>\;=\;\<-KVK+KVKSK+KSKVK,U\>.
$$
Thus, if $U=V$ and $S=\tfrac{n}{2} I_n$
$$\nabla_U\nabla_U I(S\|\Sigma)={\rm tr}(KUKUK(2S-\Sigma))={\rm tr}(KUKUK(nI_n-\Sigma)).$$
To show that this is always nonnegative for every $U\in \S^n$, we use the fact that $KUKUK$ and $nI_n-\Sigma$ are always positive semidefinite. For $nI_n-\Sigma$, this follows from the fact that the maximum eigenvalue of any correlation matrix is bounded above by ${\rm tr}(\Sigma)=n$. 
\end{proof}

\subsection{Numerical optimization of the Stein's loss}

Let $K(x)$ be the matrix equal to $W=S^{-1}$ outside the diagonal and whose diagonal is $x=(x_1,\ldots,x_n)$. By item (b) in the beginning of this section we get that the unique correlation matrix $\check \Sigma$ that optimizes the Stein's loss satisfies $\check K_{ij}=W_{ij}$ for $i\neq j$. In other words, $\check K=\check \Sigma^{-1}$ satisfies $\check K=K(x^*)$ for some $x^*\in \RR^n$. We now exploit this dimensionality reduction (from matrices to vectors) to construct a simple numerical scheme to optimize the Stein's loss, which provides a computationally efficient alternative to the algorithm developed in~\cite{brownlees2020projected}. 

\begin{prop}
The unique point $\check \Sigma$ optimizing the Stein's loss over the set of correlation functions satisfies $\check K=K(x^*)$, where $x^* \in \RR^n$ is the maximum of the concave function
\begin{equation}\label{eq:dualstain0}
    f(x)\;=\;\log\det K(x)-{\rm tr}(K(x)),\qquad x\in \RR^n.
\end{equation}
\end{prop}
\begin{proof}
The Lagrangian of the problem  of minimizing $I(\Sigma||S)$ over the set of correlation matrices is
$$
\cL(\Sigma,\Lambda)\;=\;\<\Sigma,W\>-\log\det(\Sigma W)-n+\<\Lambda,\Sigma-I_n\>,
$$
where $\Lambda$ is a diagonal $n\times n$ matrix of Lagrange multipliers. The dual function is $h(\Lambda)=\inf_{\Sigma\in \S^n_+} \cL(\Sigma,\Lambda)$ and the infimum is obtained for $\Sigma=(W+\Lambda)^{-1}$. After plugging this in we obtain
\begin{eqnarray*}
h(\Lambda)&=&\log\det(W+\Lambda)-\log\det W-{\rm tr}(\Lambda)\\
&=& \log\det(W+\Lambda)-{\rm tr}(W+\Lambda)-(\log\det W-{\rm tr}(W)).
\end{eqnarray*}
The convex dual problem is given by maximizing $h(\Lambda)$ over all diagonal matrices $\Lambda$. A simple reparametrization $K(x)=W+\Lambda$ concludes the proof.
\end{proof}

The optimization problem of maximizing (\ref{eq:dualstain0}) can be very reliably done using gradient descent, which can be complemented with Newton's method after some burn-in steps. Denoting $\Sigma(x)=(K(x))^{-1}$ we get 
\begin{equation}\label{eq:nablaKx}
\nabla f(x)\;=\;{\rm diag}(\Sigma(x))-\mathbf 1 \quad \mbox{ and } \quad
\nabla^2 f(x)\;=\;-\Sigma(x)\circ\Sigma(x),
\end{equation}
where $\circ$ denotes the Hadamard product. Our basic  implementation of this can easily handle $n$ in the hundreds. 

Finally we note that, alternatively, a simple coordinate descent scheme is possible. For any $A\in \S^n$ denote by $A_{\setminus i,i}$ the vector $\RR^{n-1}$ whose elements are $A_{ij}$ for $j\neq i$. Similarly, $A_{\setminus i}$ denotes the matrix in $\S^{n-1}$ obtained from $A$ by removing its $i$-th row and $i$-th column.
\begin{lemma}
The maximum of $f(x)$ with respect to $x_i$ with other entries of $x$ fixed is
$$
x_i^* \;=\; 1+W_{i,\setminus i} (K_{\setminus i}(x))^{-1} W_{\setminus i,i}.
$$
\end{lemma}
\begin{proof}
By (\ref{eq:nablaKx}), $\frac{\partial f}{\partial x_i}(x)=0$ if and only if $\Sigma_{ii}(x)=1$. Since $\Sigma_{ii}(x)=\tfrac{\det K_{\setminus i}(x)}{\det K(x)}$, equivalently $\det K_{\setminus i}(x)=\det K(x)$. By a standard Schur complement argument, 
$$\det K(x)=(K_{ii}(x)-K_{i,\setminus i}(x)(K_{\setminus i})^{-1}K_{\setminus i,i})\det K_{\setminus i}(x).$$
Since $K_{ii}(x)=x_i$ and $K_{i,\setminus i}(x)=W_{i,\setminus i}$ the corresponding partial derivative vanishes if $x_i-W_{i,\setminus i} (K_{\setminus i}(x))^{-1} W_{\setminus i,i}=1$.
\end{proof}
Note that since $f$ is a smooth and strictly concave, it is clear that this coordinate descent procedure converges.

\section{Equicorrelation model}\label{sec:equicorr}

An interesting example of a tractable linear correlation models is the equicorrelation model, which consists of positive definite matrices for which $\Sigma_{ij}=\rho$ for all $i\neq j$. In other words $\Sigma=(1-\rho)I_n+\rho\mathbf 1\mathbf 1^T$, where $\mathbf 1\in \RR^n$ is the vector of ones and so $\cL={\rm span}\{\mathbf 1\mathbf 1^T-I_n\}$ in (\ref{eq:corrmodel}). We denote this model by $E_n$. Note that $E_2$ is precisely the bivariate correlation model studied in Section~\ref{sec:bivariate}.

The common off-diagonal entry $\rho$ for every matrix in $E_n$ necessarily satisfies $-\tfrac{1}{n-1}<\rho<1$ for $\Sigma=(1-\rho)I_n+\rho\mathbf 1\mathbf 1^T$ to be positive definite. This is a consequence of the following well-known result.

\begin{lemma}\label{lem:basicab}
If $A=(x-y)I_n+y\mathbf 1 \mathbf 1^T$ then $\det A=(x-y)^{n-1}(x+y(n-1))$ and $A\in \S^n_+$ if and only if 
$$
x>0\qquad\mbox{and}\qquad -\frac{x}{n-1}<y<x.
$$
Furthermore, if $\det(A)\neq 0$ then 
$$
A^{-1}\;=\;\frac{1}{x-y}\left(I_n-\frac{y}{x+(n-1)y}\mathbf 1\mathbf 1^T\right).
$$
\end{lemma}

In this way, if $\Sigma=(1-\rho)I_n+\rho\mathbf 1\mathbf 1^T$ then from Lemma~\ref{lem:basicab} we obtain that  $K=(c-d)I_n+d\mathbf 1\mathbf 1^T$, with
\begin{equation}\label{eq:cd}
c=\frac{1}{1-\rho}\frac{1+(n-2)\rho}{1+(n-1)\rho},\qquad d=-\frac{1}{1-\rho}\frac{\rho}{1+(n-1)\rho}.
\end{equation}

\subsection{Algebraic degrees}

In this section we consider the problem of maximizing the entropy loss, the Stein's loss, and the symmetrized Stein's loss for the equicorrelation model $E_n$. We start with a useful result that exploits the symmetry of the model to simplify our calculations. Define the symmetrized versions of $S$ and $W=S^{-1}$ as
$$
\overline S\;=\;\frac{1}{n!}\sum_{P\in \cS_n} P S P^T,\qquad \overline W\;=\;\frac{1}{n!}\sum_{P\in \cS_n} P W P^T,
$$
where $\cS_n$ is the set of all $n\times n$ permutation matrices. The crucial observation is that the equicorrelation model is invariant under the permutation group action.
\begin{lemma}\label{lem:Sbar}
For every $S\in \S^n$ and $\Sigma\in E_n$ it holds that
$$
\<S,\Sigma^{-1}\>=\<\overline S,\Sigma^{-1}\>\qquad\mbox{and}\qquad \<\Sigma,W\>=\<\Sigma,\overline W\>.
$$
\end{lemma}
\begin{proof}
If $\Sigma\in E_n$ then $P\Sigma P^T=\Sigma$ and $P^T\Sigma^{-1}P=\Sigma^{-1}$ for every $P\in \cS_n$, so that
$$
\<S,\Sigma^{-1}\>=\frac{1}{n!}\sum_{P\in \cS_n}\<S,P^T\Sigma^{-1}P\>=\frac{1}{n!}\sum_{P\in \cS_n}\<PSP^T,\Sigma^{-1}\>=\<\overline S,\Sigma^{-1}\>.
$$
The equality $\<\Sigma,W\>=\<\Sigma,\overline W\>$ is argued in the same way.
\end{proof}
Lemma~\ref{lem:Sbar} immediately implies that optimizing the entropy loss $I(S\|\Sigma)$ is equivalent to optimizing the entropy loss $I(\overline S\|\Sigma)$ (the difference between the two does not depend on $\Sigma$). Similarly, optimizing $I(\Sigma\|S)$ is equivalent to optimizing $I(\Sigma\|\overline W^{-1})$. Moreover, 
$$L(\Sigma,S)\;=\;\tfrac{1}{2}(\<\overline S,\Sigma^{-1}\>+\<\Sigma,\overline W\>) - n\qquad\mbox{for all }S\in \S^n_+\mbox{ and } \Sigma\in E_n.$$

Let $\bar c, \bar d$ denote the diagonal and the off-diagonal entry of $\overline W$ so that $\overline W=(\bar c-\bar d)I_n+\overline d\mathbf 1\mathbf 1^T$. Recall that $c,d$ are the corresponding entries of $K$.
\begin{thm}\label{th:dMLequi}
The dual ML degree in the equicorrelation model $E_n$ is 2 for any $n\geq 2$. If $\bar d=0$ then the unique critical point is $\check\rho =0$. If $\bar d\neq 0$, there are two critical points given by 
\begin{equation}\label{eq:dualSigequicorr}
 \check{\rho}\; =\; \frac{1+(n-2)\bar d \pm \sqrt{(n \bar d+1)^2-4\bar d}}{2(n-1)\bar d}.
\end{equation}
The unique positive definite solution $\check \Sigma$ corresponds to the solution that takes the negative sign in (\ref{eq:dualSigequicorr}). Moreover, $\check K=\check \Sigma^{-1}$ is equal to $\bar d$ outside of the diagonal and all its diagonal entries are equal and given by taking the positive sign in
\begin{equation}\label{eq:dualKequicorr}
\frac{1}{2}\left(1-(n-2)\bar d\pm \sqrt{(n\bar d+1)^2-4\bar d}\right).
\end{equation}
\end{thm}
\begin{proof}
For the Stein's loss the condition that the critical point $\check K$ must satisfy is $\<\check K-W,\mathbf 1\mathbf 1^T-I_n\>=0$, or in other words, the sums of all the off-diagonal elements of both $\check K$ and $W$ are equal. Using (\ref{eq:cd}) the critical equation becomes
\begin{equation}\label{eq:quadratic}
-\frac{1}{1-\rho}\frac{\rho}{1+(n-1)\rho}\;=\;\bar d.
\end{equation}
If $\bar d=0$ there is a unique solution given by $\rho=0$. If $\bar d\neq 0$, this leads to a quadratic equation in $\rho$: 
$$
    (n-1)\bar d \rho^2 -( (n-2)\bar d + 1) \rho - \bar d = 0
$$
with roots satisfying (\ref{eq:dualSigequicorr}). We can rewrite \eqref{eq:quadratic} as $-\rho = \bar d(1-\rho)(1+(n-1)\rho)$. Geometrically, this corresponds to intersecting the line $-\rho$ with the parabola $\bar d(1-\rho)(1+(n-1)\rho)$. If $\bar d>0$, the intersection point with larger $\rho$ coordinate has $\rho>1$ and thus does not lead to a positive definite solution. The other solution clearly satisfies $-\tfrac{1}{n-1}<\rho<0$. This point is obtained by choosing the negative sign in \eqref{eq:dualSigequicorr}. Similarly, if $\bar d<0$ the intersection point with the larger $\rho$ coordinate satisfies $0<\rho<1$  and the other intersection point satisfies $\rho<-\tfrac{1}{n-1}$. The former is given by choosing the negative sign in \eqref{eq:dualSigequicorr} and the latter by the positive sign. We conclude that in both cases we should take the negative sign in \eqref{eq:dualSigequicorr} to guarantee $-\tfrac{1}{n-1}<\rho<1$. The last part of the statement follows from the formula for the off-diagonal entry of $\Sigma^{-1}$ by taking $\rho=\check \rho$.
\end{proof}

For the entropy loss, we have no guarantee that there is a unique positive definite solution and the situation becomes slightly more complicated, but also more interesting. 

Let $a,b$ be the diagonal and the off-diagonal entry of $\overline S$, so that  $\overline S=(a-b)I_n+b\mathbf 1\mathbf 1^T$. Using Lemma~\ref{lem:Sbar}, to optimize $I(S\|\Sigma)$ over the equicorrelation model $E_n$, we can equivalently optimize
\begin{equation}\label{eq:entrequi}
I(\overline S\|\Sigma)\;=\;(n-1)\log\left(\tfrac{1-\rho}{a-b}\right)+\log\left(\tfrac{1+\rho(n-1)}{a+b(n-1)}\right)+n(ac+bd(n-1)-1),
\end{equation}
where $c,d$ are the entries of $K=\Sigma^{-1}$ as given in (\ref{eq:cd}). Recall that  $\bar c,\bar d$ denote the entries of $\overline W=(\bar c-\bar d)I_n+\bar d\mathbf 1\mathbf 1^T$. The next result provides both the ML degree and the SSL degree we had promised at the end of Section \ref{sec:bivariate}.

\begin{thm}\label{th:MLESigequicorr}
The ML degree of the equicorrelation model $E_n$ is 3 for any $n\geq 2$. The critical points are given by solutions of the cubic equation $f_n(\rho)=0$, where
\begin{equation}\label{eq:MLESigequicorr}
   f_n(\rho)\;:=\; (n-1)\rho^3 + ((n-2)(a-1)-(n-1)b)\rho^2 + (2a-1)\rho - b.
\end{equation}
The SSL degree of $E_n$ is 4 for any $n\geq 2$. The critical points satisfy the quartic
\begin{eqnarray*}
    (n-1)^2 \bar{d}\rho^4 -2 (n^2-3n+2)\bar{d}\rho^3 +  ((n^2-6n+6)\bar{d} +(n-2)a-(n-1)b)\rho^2 \\ + 2(a+(n-2)\bar{d})\rho + \bar{d} - b = 0.
\end{eqnarray*}
\end{thm}
\begin{proof}If $K=(c-d)I_n+d\mathbf 1\mathbf 1^T$ and $\overline S=(a-b)I_n+b\mathbf 1\mathbf 1^T$ then the matrix $K\overline{S} K$ has the  form $(x-y)I_n+y\mathbf 1\mathbf 1^T$ with 
$$
y=a d(2c+(n-2)d)+b\left((c+(n-2)d)^2+(n-1)d^2\right).
$$
The critical point of $I(\overline S\| \Sigma)$ must satisfy that $y=d$. Using \eqref{eq:cd}, this translates to a cubic equation (\ref{eq:MLESigequicorr}). To compute the SSL degree, note that, by Lemma~\ref{lem:Sbar}, $L(\Sigma,S)=\tfrac{1}{2}(\<\overline S,\Sigma^{-1}\>+\<\Sigma,\overline W\>)-n$. Since $\bar c, \bar d$ are the entries of $\overline W$, we have
$$
L(\Sigma,S)=\frac{n}{2}\left(\bar c+a c+(n-1)(\bar d \rho+b d)\right)-n.
$$
The equation $\frac{\partial}{\partial \rho}L(\Sigma,S)=0$ is equivalent to the vanishing of the given quartic polynomial in $\rho$.
\end{proof}

\subsection{Positive definite critical points}

Recall that the ML degree is three and thus the cubic $f_n(\rho)=0$ in \eqref{eq:MLESigequicorr} always has a real root. Note that  $f_n(-1)=-\tfrac{n}{(n-1)^2}(a+(n-1)b)$ and $f_n(1)=n(a-b)$. Since $\overline S\in \S^n_+$, by Lemma~\ref{lem:basicab} we have that $f_n(1)>0$ and $f_n(-1)<0$. This means that in fact $f_n$ always has a root in $(-\tfrac{1}{n-1},1)$ which corresponds to a positive definite solution. It has precisely one real root if its discriminant
\begin{eqnarray}\label{eq:bigdiscriminant}
    \Delta_{f,n}(b,a) &=& -4(n-1)^3b^4 + 12(n-2)(a-1)b^3 \nonumber \\ & &-4(n-1)((3n^2-13n+13)a^2-2a(3n^2-8n+8)+3n^2-n+1)b^2 \nonumber \\
& & +4(n-2)(a-1)((n^2-6n+6)a^2-(2n^2-n+1)a+n^2)b \nonumber \\
& & +(2a-1)^2((n-2)^2a^2-2n^2a+n^2) 
\end{eqnarray}
is negative. Conversely, if $\Delta_{f,n} > 0$ there are three real roots and if $\Delta_{f,n} = 0$ one of them is a double root. Note that if $n=2$ then $\Delta_{f,n}$ is precisely the discriminant $\Delta_{f}$ in (\ref{fig:realsolutions}) and the underlying geometry was carefully explained in Section~\ref{sec:bivariate}. To get a complete picture for how multiple roots may arise in this optimization problem, we follow the same strategy as for the bivariate model. The main difference is that the region is no longer symmetric about the $a$-axis since the left boundary $a = -(n-1)b$ moves toward it; see Figure~\ref{fig:largen}.

\begin{figure}[ht]
 \centering 
 \includegraphics[width=8cm]{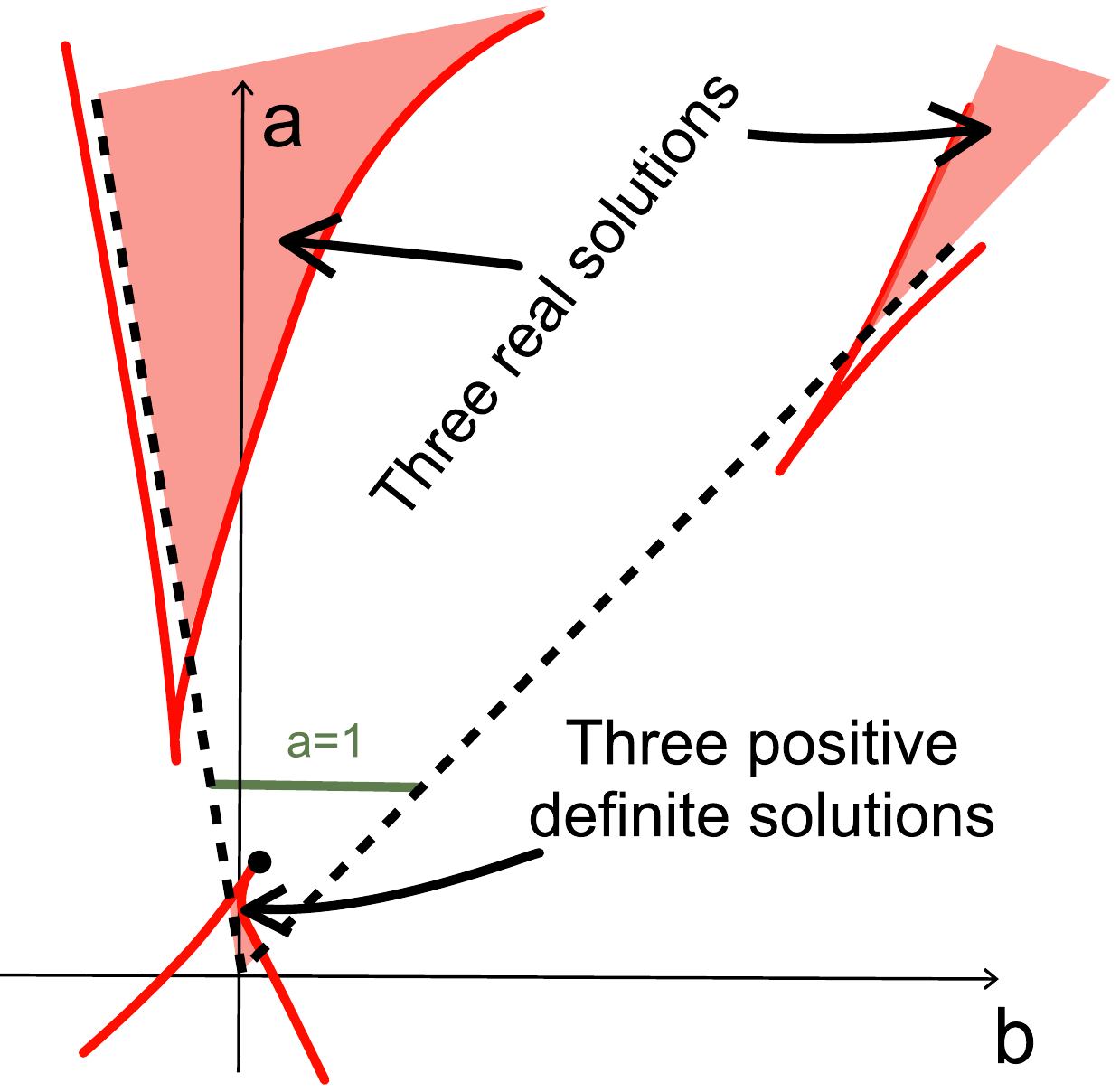} 
\caption{Likelihood geometry for $n>2$. The set of all points in the positive definite cone such that $a=1$ is marked in green. }\label{fig:largen} 
\end{figure}

One stark difference from the $n=2$ case that can be appreciated in Figure~\ref{fig:largen} is that for $n>2$ the discriminant $\Delta_{f,n}(b,a)$ crosses the $a$-axis three times, not only at the point $(0,\frac12)$ but also when $b = \frac{n(n \pm 2\sqrt{n-1})}{(n-2)^2}$. Note however that both of these extra points converge to 1 as $n\to \infty$.

The discriminant intersects the boundary of the positive definite cone in four points. On the $a=b$ side these are given by $b = \pm 2\sqrt{n(n-1)}+2n-1$
(note that as $n \to \infty$, one goes to 0 and the other goes to $\infty$)
while on the $a=-(n-1)b$ side these are:
$b = \frac{1}{(\sqrt{n}\pm 1)^2}.$
This
means that as $n\to \infty$ these points converge to $(b,a) = (0, 1)$. One of the consequences is that the red chamber in the bottom shrinks in area towards zero.

In summary, we observe the following pattern for every $n\geq 2$. Inside the positive definite cone there are four chambers induced by the discriminant locus (the non-central chambers are marked in red in Figure~\ref{fig:largen}). The large central chamber leads to only one real critical point of the entropy loss. On the other hand, if $(b,a)$ lies in one of the other chambers, then there are 3 real critical points in the entropy loss. However, in the two side red chambers there are two extra critical points outside the positive definite cone (that is, there is a unique positive definite solution to the likelihood equations) and only the bottom chamber (which shrinks as $n$ grows) induces all three positive definite solutions.

\begin{figure}[ht]
 \centering 
 \includegraphics[scale=.7]{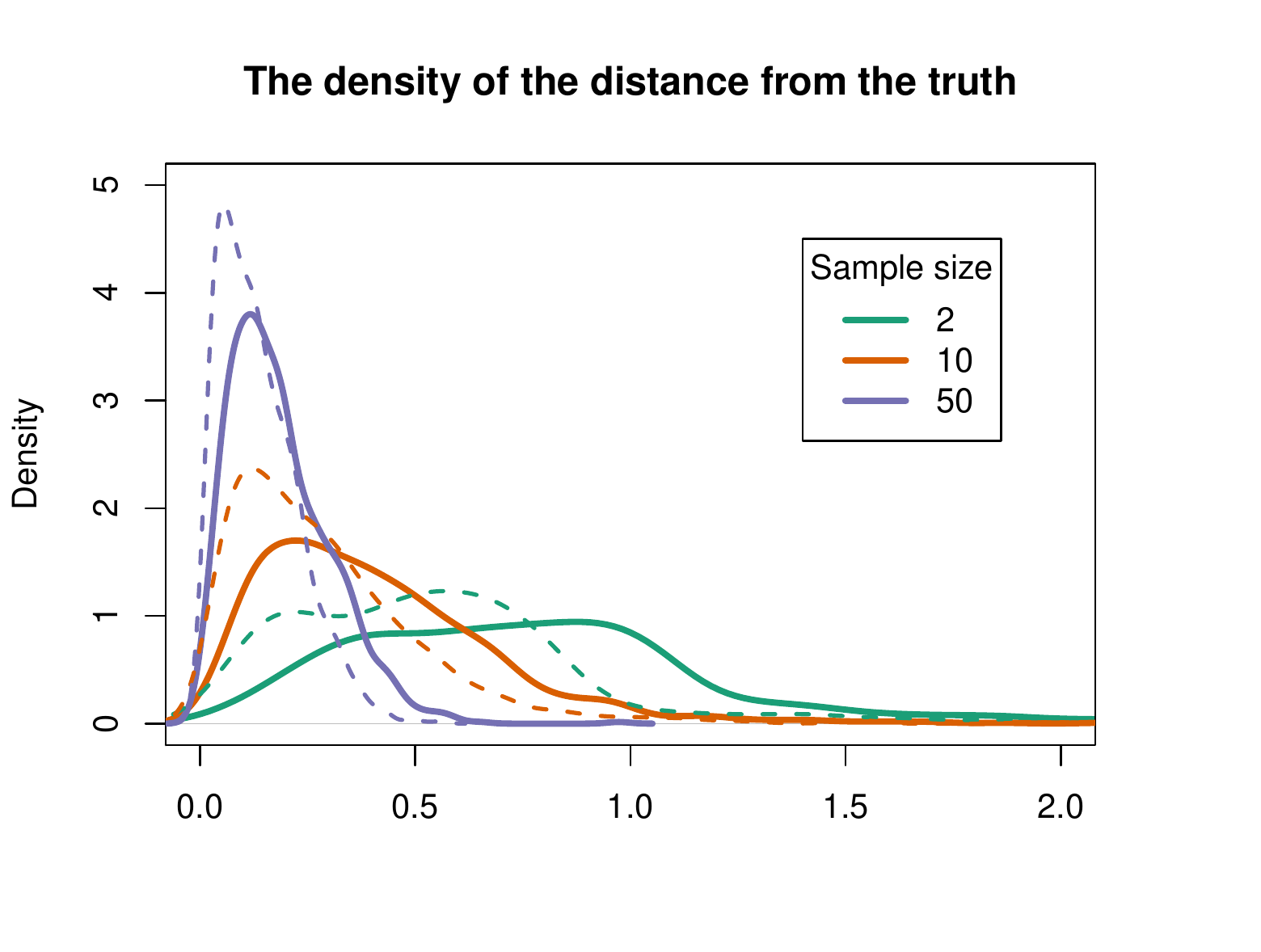}
\caption{Density plots for the Euclidean distance of $(b,a)$ from $(\rho^*,1)$ where $\rho^*=0.5$ for $n=2$ (solid lines) and $n=10$ (dashed lines). In both cases we consider sample sizes, $2,10, 50.$ }\label{fig:dist}
\end{figure}
\begin{rem}\label{rem:probrelev}
From the mathematical perspective we would like to understand the geometry  of the entropy  loss for every pair $b,a$ for which the entropy loss is defined. However, it is important to note that statisticians will typically care only about a small neighborhood of $(b^*,a^*)=(\rho^*,1)$, where $\rho^*$ is the true data generating correlation. To illustrate this point, we show in Figure~\ref{fig:dist} the distribution of the  distance of $(b,a)$ from $(\rho^*,1)$ assuming that the data follow the Gaussian distribution with covariance in $E_n$ and that $\rho^*=0.5$. Even for small $n$ and relatively small sample sizes this distance is small with a very high probability. A similar behavior will be observed for any sub-Gaussian distribution. 
\end{rem}

Remark~\ref{rem:probrelev} has an interesting consequence for our analysis. Although the entropy loss can have three positive definite critical points, in statistical practice this will happen rarely even for small sample sizes; see Figure~\ref{fig:prob}.

\begin{figure}[ht]
 \centering 
 \includegraphics[scale=.5]{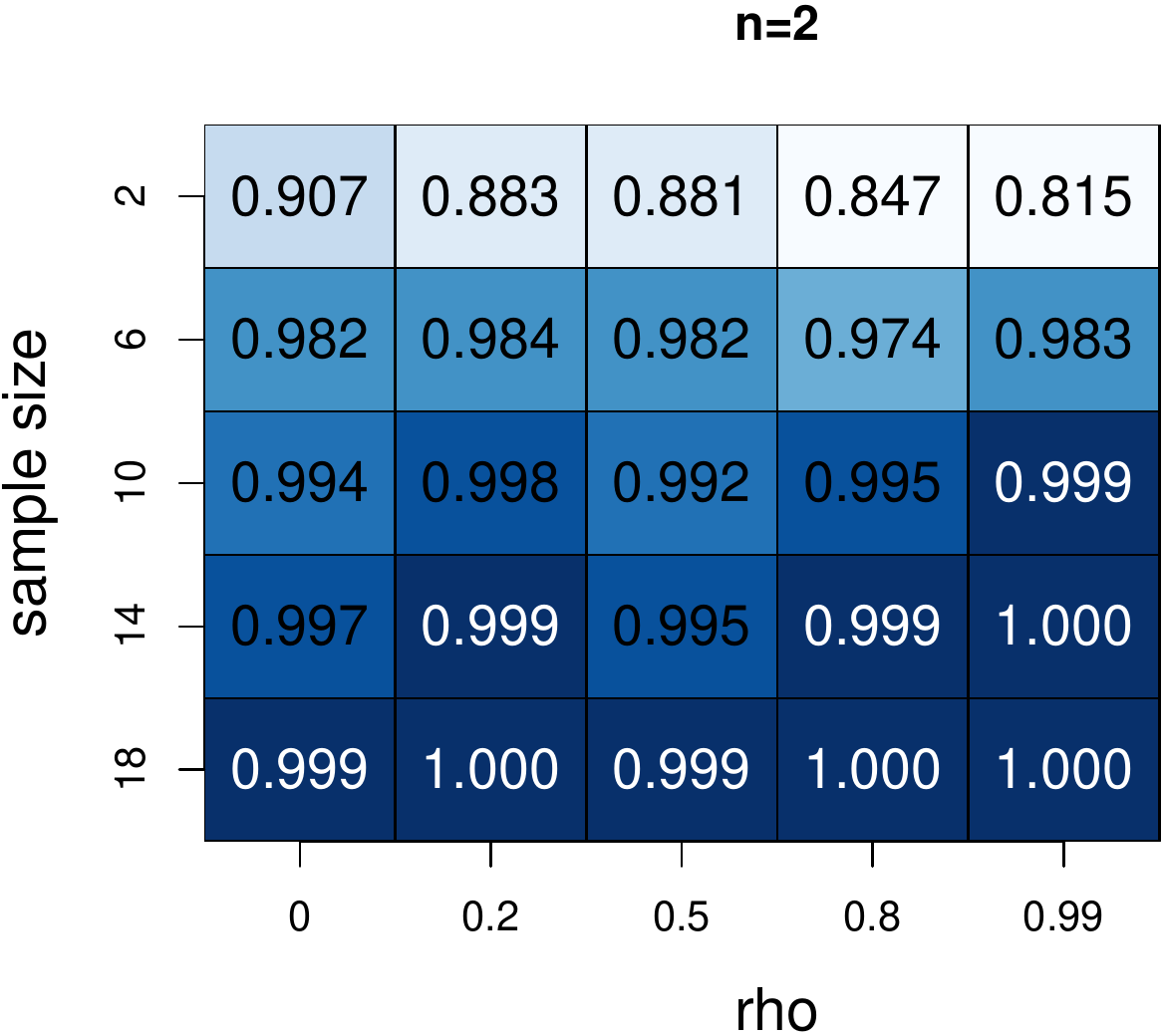}\quad\includegraphics[scale=.5]{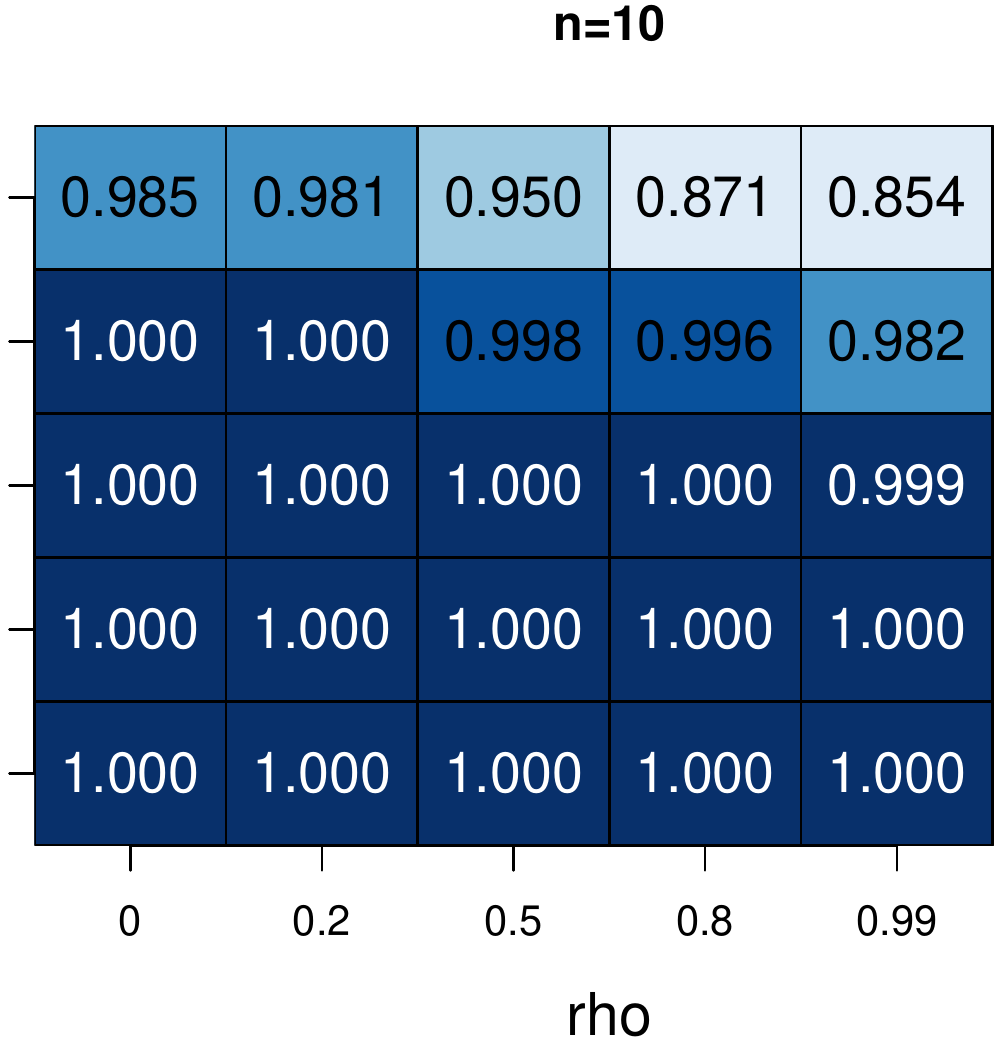}
\caption{The probability that the entropy loss in the equicorrelation model $E_n$ has one critical point or equivalently, the probability that $\Delta(b,a)< 0$, as a function of the true correlation and the sample size. }\label{fig:prob}
\end{figure}

\subsection{Convexity}

In the equicorrelation model $E_n$, the entropy loss given in (\ref{eq:entrequi}) depends on $S$ only through the entries $a,b$ of the symmetrized matrix $\overline S$ (up to a term that does not depend on $\rho$). To describe the convexity cone $C(E_n)$ (c.f. Definition~\ref{df:convcone}) it is then enough to find for which matrices $\overline S$ the entropy loss is convex. We do it by studying the second order condition for convexity. The second derivative of the entropy loss with respect to $\rho$ is equal to $$\tfrac{n(n-1)}{(1-\rho)^3(1+(n-1)\rho)^3}g_n(\rho),$$ where
\begin{eqnarray}
g_n(\rho)&=&2a-1  + 2b(n-2) - ( 6b(n-1) + n-2)\rho + 6a(n-1)\rho^2  \\
\nonumber &+ & ((2a-1)(n-2) - 2b(n-1))(n-1)\rho^3 + (n-1)^2\rho^4
\end{eqnarray}
The entropy loss is convex in the positive definite cone, which by Lemma~\ref{lem:basicab} consists of equicorrelation matrices with $\rho\in (-\tfrac{1}{n-1},1)$, if and only if $g_n(\rho)\geq 0$ in this interval.

Since $g_n(-\tfrac{1}{n-1})=\tfrac{2n^2(a+b(n-1))}{(n-1)^2}$ and $g_n(1)=2(a-b)n^2$ are positive in $\S^n_+$, we believe that a similar analysis to the one for the bivariate case in Section~\ref{sec:bivariate} is possible. Alternatively, a simple lower bound for $C(E_n)$ can be easily obtained without studying the discriminant $\Delta_{g,n}$ of $g_n$. 
\begin{lemma}\label{lem:concave}
In the equicorrelation model $E_n$, if the entries $a$ and $b$ of $\overline{S}$ satisfy
\begin{equation}\label{eq:concavecondition}
-\frac{a}{n-1} + \frac{n}{2(n-1)} \leq b \leq  a - \frac{n}{2(n-1)}
\end{equation}
then the entropy loss function is strictly convex in the positive definite cone. 
\end{lemma}
\begin{proof}
We use the result from \cite[Proposition~3.1]{ZUR} which states that whenever $2\overline S - \Sigma \succ 0$, the conclusion follows. In our case, the diagonal entries of $2\overline S - \Sigma$ are $2a-1$ and the off-diagonal entries equal $2b-\rho$. Hence, by Lemma~\ref{lem:basicab}, $2\overline S - \Sigma \succ 0$ if and only if
$2a>1$ and $- \frac{2a-1}{n-1} < 2b - \rho < 2a-1$, 
or equivalently,
\begin{equation*}
2a>1\qquad\mbox{and}\qquad 1 + 2(b-a) < \rho < 2b + \frac{2a-1}{n-1}.
\end{equation*}
Since $-\frac{1}{n-1} < \rho < 1$ always holds for all $\rho$ in the correlation model, to conclude that the function is globally convex, we need that $2b + \frac{2a-1}{n-1}\geq 1$ and $1 + 2(b-a)\leq -\tfrac{1}{n-1}$. We easily show that this is equivalent to  $-\frac{a}{n-1} + \frac{n}{2(n-1)} \leq b \leq  a - \frac{n}{2(n-1)}$. 
\end{proof}

\section{Further examples and open questions}\label{sec:further}

Our study of the likelihood geometry of linear correlation models leads to several new insights both for linear correlation models and general affine covariance models. We believe that some of them can be further explored from an statistical as well as an algebraic perspective. In what follows we offer some guidance on potential open questions that we found particularly interesting.

\paragraph{The convexity cone}
One of the new objects  to study is the convexity cone, which we showed to be always non-empty for every linear correlation model. It would be useful to describe the convexity cone for the unconstrained correlation model. Another question is whether we can describe the convexity cone of one-dimensional linear correlation models directly in terms of their generator. It follows from the proof of Proposition~\ref{prop:convnonempty} that the convexity cone of an affine covariance model $M$ is always non-empty when we can uniformly bound the maximal eigenvalue of the matrices $\Sigma\in M$. It would be interesting to see if this eigenvalue condition is necessary for non-emptiness of the convexity cone.

\paragraph{The symmetrized Stein's loss and other losses}

In this paper we proposed the study of the SSL degree, the number of complex solutions to the critical equations of the symmetrized Stein's loss. In all our calculations we observed that the dual ML degree is smaller than the ML degree, which is then smaller than the SSL degree. If these inequalities hold for general models, this would indicate that, despite its algebraicity, the symmetrized Stein's loss may be relatively harder to handle from an algebraic point of view. Currently we do not know how to prove this.

\paragraph{Data leading to the same critical point} One interesting aspect of the likelihood geometry, which we did not study in detail here, concerns the geometry of the data matrices $S$ that lead to the same critical point $\Sigma^*$. Directly from the critical equations in Theorem~\ref{th:critical} we get the following result.
\begin{prop}\label{prop:linearextr}
Let $\Sigma^*$ be a fixed point in an affine covariance model $M=A+\cL$. The set of all $S\in \S^n_+$ for which $I(S\| \Sigma)$ has $\Sigma^*$ as one of its critical points is a linear subspace of $\S^n_+$. 
\end{prop} 
\begin{proof}
If $U_1,\ldots,U_k$ are the generators of $\cL$ then $\Sigma^*$ is a critical point of $I(S\| \Sigma)$ for data $S$ if and only if $\<K^*-K^*SK^*,U_i\>=0$ for $i=1,\ldots,k$. This is a linear condition in $S$.
\end{proof}
The closely related study of all $S$ that lead to the same global minimum of the entropy loss is more subtle. To illustrate, consider the bivariate correlation model in Section~\ref{sec:bivariate}. Fix $\Sigma^*$ in the model with an underlying correlation $\rho$. Then the linear space from Proposition~\ref{prop:linearextr} is given in $(b,a)$ by the equation $f(\rho)=0$ (linear in $a,b$), that is,
$\rho^3-b\rho^2-\rho(1-2a)-b\;=\;0.$
Consider any point $(b,a)$ on this line that does not lie in the bottom red region in Figure~\ref{fig:realsolutions}. By Theorem~\ref{th:bivariatecritical} every such point leads to a unique critical point of the entropy loss, and thus the fixed $\rho$ is the optimal point for the data $(b,a)$. However, if $(b,a)$ lies in the bottom red region then $\rho$ may not be the optimum, so careful analysis is needed to understand the full picture.

We note that Proposition~\ref{prop:linearextr} extends to the Stein's loss: the set of all $W=S^{-1}\in \cS^n_+$ for which $I(\Sigma\| S)$ has $\Sigma^*$ as one of its critical points forms a linear subspace of $\cS^n_+$.

\paragraph{Other interesting correlation models}
There are certainly other tractable and interesting linear correlation models that one can consider. A one-dimensional example is given by the tridiagonal equicorrelation model, where $\Sigma$ is a correlation matrix satisfying $\Sigma_{ij}=\rho$ if $|i-j|=1$, and $\Sigma_{ij}=0$ for all other pairs $i\neq j$. In this case the linear subspace $\cL$ defining this model is generated by a single element $U$ such that $U_{ij}=1$ if $|i-j|=1$ and $U_{ij}=0$ otherwise. 

If $n=3$, the results are as in the equicorrelation case: the three algebraic degrees are 2,3,4. It is a natural question whether there is a deeper reason for this. If $n=4$, the dual ML degree is 4, the ML degree is 7, and the SSL-degree is 8. The dual ML degree sequence is actually $2\lfloor \tfrac{n}{2}\rfloor$. To prove this fact we use the following auxiliary result. 
\begin{lemma}\label{lem:symmequi}
The determinant $\delta_n$ of the symmetric tridiagonal matrix of dimension $n\geq 2$, with 1's on the diagonal and $\rho$ in the upper/lower diagonal, satisfies
$$
\delta_n\;=\;\sum_{i=0}^{\lfloor n/2\rfloor} (-1)^i {\binom{n-i}{i}}\rho^{2i},
$$
where in addition we set $\delta_0=\delta_1=1$. Moreover, 
$$
\det(\Sigma) K_{ij}\;=\;(-1)^{i+j}\rho^{j-i} \delta_{i-1}\delta_{n-j} \qquad\mbox{for }i<j.
$$
\end{lemma}
\begin{proof}
This can be inferred from the formulas (5) and (9) in \cite{hu1996analytical} noting that their matrix $M_n$ links to ours via $\Sigma=\rho M_n$ with their $D$ equal to $\tfrac{1}{\rho}$.
\end{proof}
\begin{thm}\label{th:tridual}
The dual ML degree in the tridiagonal equicorrelation model is $2\lfloor\tfrac{n}{2}\rfloor$ for any $n\geq 1$. The critical points are the solutions to the polynomial:
\begin{equation}\label{eq:dMLEtri}
s \delta_n+\rho\sum_{i=1}^{n-1} \delta_{i-1}\delta_{n-i-1}\;=\;0,
\end{equation}
where $s=\sum_{i=1}^{n-1} W_{i,i+1}$ and $\delta_n$ is as in Lemma~\ref{lem:symmequi}. 
\end{thm}
\begin{proof}
The critical equation is that $s$ is equal to $\sum_{i=1}^{n-1} K_{i,i+1}$. Using Lemma~\ref{lem:symmequi}, this is equivalent to (\ref{eq:dMLEtri}). It remains to show that the polynomial in (\ref{eq:dMLEtri}) has degree $2\lfloor\tfrac{n}{2}\rfloor$ as long as $s\neq 0$. The degree of $\rho \delta_{i-1}\delta_{n-i-1}$ is $2\left(\lfloor\tfrac{i-1}{2}\rfloor+\lfloor\tfrac{n-i-1}{2}\rfloor\right)+1$, which is \emph{strictly} less than the degree of $\delta_n$. To show this, we consider four cases depending on the parity of $n$ and $i$. If $n$ is odd then the degree of $\rho \delta_{i-1}\delta_{n-i-1}$ is $n-2$, which is less than $\deg(\delta_n)=n-1$. If $n$ is even then $\deg(\rho \delta_{i-1}\delta_{n-i-1})\in \{n-1,n-3\}$, which again is less than $\deg(\delta_n)=n$. This means that the polynomial (\ref{eq:dMLEtri}) in $\rho$ has degree $2\lfloor n/2\rfloor$, with leading coefficient
\begin{equation*}
   s (-1)^{\lfloor n/2\rfloor}{\binom{n-\lfloor n/2\rfloor}{\lfloor n/2\rfloor}}. \qedhere
\end{equation*} 
\end{proof}

\small
\section*{Acknowledgements}

CA was partially supported by the Deutsche Forschungsgemeinschaft (DFG) in the context of the Emmy
Noether junior research group KR 4512/1-1. PZ was supported from the Spanish Government grants (RYC-2017-22544,PGC2018-101643-B-I00), and Ayudas Fundaci\'on BBVA a Equipos de Investigaci\'on Cientifica 2017. We thank anonymous referees for their comments and suggestions which helped us improve the exposition.


\vspace{-1cm}
\begin{thebibliography}{10}

\bibitem{barndorff1994inference}
{\sc O.~E. Barndorff-Nielsen and D.~R. Cox}, {\em Inference and asymptotics},
  Chapman$\backslash$\& Hall, 1994.

\bibitem{boege2020reciprocal}
{\sc T.~Boege, J.~I. Coons, C.~Eur, A.~Maraj, and F.~R{\"o}ttger}, {\em
  Reciprocal maximum likelihood degrees of {B}rownian motion tree models},
  arXiv:2009.11849,  (2020).

\bibitem{brownlees2020projected}
{\sc C.~T. Brownlees and J.~Llorens-Terrazas}, {\em Projected dynamic
  conditional correlations}, Available at SSRN 3576985,  (2020).

\bibitem{coons2020maximum}
{\sc J.~I. Coons, O.~Marigliano, and M.~Ruddy}, {\em Maximum likelihood degree
  of the two-dimensional linear {G}aussian covariance model}, Algebraic
  Statistics, 11 (2020), pp.~107--123.

\bibitem{dhillon2008matrix}
{\sc I.~S. Dhillon and J.~A. Tropp}, {\em Matrix nearness problems with
  {B}regman divergences}, SIAM Journal on Matrix Analysis and Applications, 29
  (2008), pp.~1120--1146.

\bibitem{efron1978geometry}
{\sc B.~Efron}, {\em The geometry of exponential families}, The Annals of
  Statistics, 6 (1978), pp.~362--376.

\bibitem{pencils}
{\sc C.~Fevola, Y.~Mandelshtam, and B.~Sturmfels}, {\em {Pencils of Quadrics:
  Old and New}}.
\newblock arXiv:2009.04334.

\bibitem{higham2002computing}
{\sc N.~J. Higham}, {\em Computing the nearest correlation matrix—a problem
  from finance}, IMA journal of Numerical Analysis, 22 (2002), pp.~329--343.

\bibitem{hu1996analytical}
{\sc G.~Hu and R.~F. O'Connell}, {\em Analytical inversion of symmetric
  tridiagonal matrices}, Journal of Physics A: Mathematical and General, 29
  (1996), p.~1511.

\bibitem{kendall1961advanced}
{\sc M.~Kendall and A.~Stuart}, {\em The advanced theory of statistics:
  Inference and relationship, vol. 2}, London: Charles Griffin,  (1961).

\bibitem{kulis2009low}
{\sc B.~Kulis, M.~A. Sustik, and I.~S. Dhillon}, {\em Low-rank kernel learning
  with {B}regman matrix divergences.}, Journal of Machine Learning Research, 10
  (2009).

\bibitem{ledoit2018optimal}
{\sc O.~Ledoit and M.~Wolf}, {\em Optimal estimation of a large-dimensional
  covariance matrix under {S}tein’s loss}, Bernoulli, 24 (2018),
  pp.~3791--3832.

\bibitem{liang1986longitudinal}
{\sc K.-Y. Liang and S.~L. Zeger}, {\em Longitudinal data analysis using
  generalized linear models}, Biometrika, 73 (1986), pp.~13--22.

\bibitem{moakher2006symmetric}
{\sc M.~Moakher and P.~G. Batchelor}, {\em Symmetric positive-definite
  matrices: From geometry to applications and visualization}, in Visualization
  and Processing of Tensor Fields, Springer, 2006, pp.~285--298.

\bibitem{rousseeuw1994shape}
{\sc P.~J. Rousseeuw and G.~Molenberghs}, {\em The shape of correlation
  matrices}, The American Statistician, 48 (1994), pp.~276--279.

\bibitem{small2000eliminating}
{\sc C.~G. Small, J.~Wang, and Z.~Yang}, {\em Eliminating multiple root
  problems in estimation}, Statistical Science, 15 (2000), pp.~313--341.

\bibitem{soloff2020covariance}
{\sc J.~A. Soloff, A.~Guntuboyina, and M.~I. Jordan}, {\em Covariance
  estimation with nonnegative partial correlations}, arXiv:2007.15252,  (2020).

\bibitem{sturmfels2020estimating}
{\sc B.~Sturmfels, S.~Timme, and P.~Zwiernik}, {\em Estimating linear
  covariance models with numerical nonlinear algebra}, Algebraic Statistics, 11
  (2020), pp.~31--52.

\bibitem{ZUR}
{\sc P.~Zwiernik, C.~Uhler, and D.~Richards}, {\em Maximum likelihood
  estimation for linear {G}aussian covariance models}, J. R. Stat. Soc. Ser. B.
  Stat. Methodol., 79 (2017), pp.~1269--1292.

\end{thebibliography}
\end{document}